\documentclass[11pt]{amsart}

\usepackage{palatino} %1
\usepackage{pinlabel}
\usepackage{pdfpages}
\usepackage{hyperref}

\usepackage{amsmath,amssymb,amsthm,verbatim, amsfonts,amscd,flafter,epsf, epsfig,graphicx,verbatim,pinlabel,mathrsfs}
\usepackage[all]{xy}
\usepackage{epsf}
\usepackage[abs]{overpic}
\usepackage{epstopdf}
\usepackage{color}
\usepackage{mathtools}

\definecolor{red}{rgb}{1,0,0}
\definecolor{blue}{rgb}{0,0,1}
\definecolor{green}{rgb}{0,1,0}

\newtheorem{theorem}{Theorem}[section]
\newtheorem{lemma}[theorem]{Lemma}

\newtheorem{corollary}[theorem]{Corollary}
\newtheorem{proposition}[theorem]{Proposition}

\theoremstyle{definition}
\newtheorem{definition}[theorem]{Definition}

\newtheorem{remark}[theorem]{Remark}

\newtheorem{example}[theorem]{Example}

\def\F{\mathcal{F}}

\def\pit{{\widetilde{\pi}}}

   \title{A User's Guide to Morse Foliated Open Books}

\author[Vera V\'ertesi]{Vera V\'ertesi}
\address{University of Vienna}
\email{vera.vertesi@univie.ac.at}

\author[Joan E. Licata]{Joan E. Licata}
\address{Mathematical Sciences Institute, The Australian National University}
\email{joan.licata@anu.edu.au}

\begin{document}
\begin{abstract}
Morse foliated open books were introduced in \cite{LV} (along with abstract and embedded versions) as a tool for studying contact manifolds with boundary, and this article illustrates the advantages of the Morse perspective.  We use this to extend the definition of \textit{right-veering} to foliated open books and we show that it plays a similar role in detecting overtwistedness as in other versions of open books.
 
\end{abstract}
\maketitle

\maketitle

%%%%%%%%%%%%%%%%%%%%%%%%%%%%%%%%%%%%%%%%%%%%%%%%%%%%%%%

%\tableofcontents

%%%%%%%%%%%%%%%%%%%%%%%%%%%%%%%%%%%%%%%%%%%%%%%%%%%%%%%

%%%%%%%%%%%%%%%%%%%%%%%%%%%%%%%%%%%%%%%%%%%%%%%%%%%%%%%
% !TEX root = RV_main.tex
%%%%%%%%%%%%%%%%%%%%%%%%%%%%%%%%%%%%%%%%%%%%%%%%%%%%%%%
\section{Introduction}

Three flavours of foliated open books were introduced in \cite{LV}, each a topological decomposition of a manifold with boundary which determines an equivalence class of contact structures on the ambient manifold.  Embedded foliated open books provide an intuitive construction: cut a traditional open book along a generic separating surface and the result is a pair of embedded foliated open books.  Abstract foliated open books were explored further in \cite{AFHLPV} where they were used to defined a contact invariant in bordered sutured Floer homology.  In this article we turn attention to Morse foliated open books, illustrating the benefits of this perspective by extending the established notion of a right-veering monodromy to the open book setting.  

We also admit to a fondness for Morse foliated open books.  There are technical advantages to having three versions to select from, but when foliated open books existed only as chalked pictures on a board, their intrinsic data always included a circle-valued Morse function.  It is therefore a pleasure to return to this approach now. For a topologist, one of the beautiful applications of Morse theory is the metamorphosis it facilitates from differential geometry to geometric topology, turning a smooth manifold into a handle structure.  Open books of all flavours serve a similar purpose in contact geometry, encoding a non-integrable plane field via a topological decomposition.  Here, we equip the complement of a properly embedded one-manifold with an $S^1$-Morse function, all of whose critical points lie on the boundary.  As in standard Morse theory, such a critical point corresponds to a change in the topology of the level sets of the Morse function, but the level sets are now interpreted as pages of an open book decomposition.  We may thus see an unlimited number of topologically distinct page types, but the transition between any two page types is tightly controlled.  Furthermore, a catalogue of these changes is recorded on the boundary, where each critical point of the original Morse function is a critical point of its restriction to the boundary.  As previously shown, this boundary data alone determines a family of compatible Morse functions on the original contact manifold, so that a relatively small amount of data captures a broad and flexible set of decompositions.  Precise definitions and key theorems are recalled in the next section, and we briefly outline some dividends from this perspective.

The step from smooth manifold to handle decomposition requires a choice of gradient-like vector field, and such a choice is similarly useful in the case of Morse foliated open books. We choose a particular class of gradient-like vector fields characterized by their flow on the restriction to the boundary of the manifold. With this class fixed, we define the monodromy of a Morse foliated open book as the first return map of this flow relative to a fixed page.   This is defined only on a subsurface of the page, just as in the case of the partial open books defined by Honda-Kazez-Matic, and in some cases our notion of monodromy can be directly identified with their notion \cite{hkm09}.  However, the fact that foliated open books admit a quite flexible notion of pages leads to Morse foliated open books which may not immediately be interpreted as partial open books.  In this case, the monodromy of a foliated open book is a strict generalisation of the partial open book monodromy, and Section~\ref{sec:rv} explores the consonance of the two versions through a family of examples. We also show  ---perhaps unsurprisingly--- that veering monodromies play a similar role in foliated open books as in their classical and partial counterparts.

The study of right-veering monodromies initiated by Goodman and developed by Honda-Kazez-Matic relates the tightness of a contact structure to a measure of positivity in its supporting open books \cite{Goodman}, \cite{HKM09_HF}, \cite{hkm09}.  The definition of right-veering extends verbatim to foliated open book monodromies, and with similar consequences: a contact manifold is tight if and only if all of its supporting Morse foliated open books have right-veering monodromy.  In the case that a Morse foliated open book admits a left-veering arc, one may construct an overtwisted disc as in \cite{IKtot}.  

%We conclude the article with a few examples, providing explicit foliated open books for neighborhoods of overtwisted discs.  

%\begin{figure}[h!]
%\begin{center}
%\includegraphics[scale=0.8]{efobtorus}
%\caption{ Left: $M=D^2\times S^1$ in $\mathbb{R}^3$.  Center: Selected pages $S_c= M\cap \{\theta=c\}$.  Right: singular foliation on $\partial M$. The  green and blue curves in the three pictures indicate the leaves $\theta=0$ and  $\theta=\frac{\pi}{3}$, respectively. }\label{fig:efobtorus}
%\end{center}
%\end{figure}

%%%%%%%%%%%%%%%%%%%%%%%%%%%%%%%%%%%%%%%%%%%%%%%%%%%%%%%

% !TEX root = RV_main.tex
%%%%%%%%%%%%%%%%%%%%%%%%%%%%%%%%%%%%%%%%%%%%%%%%%%%%%%%
\section{Foliated Open Books}

\begin{definition} A \emph{(Morse) foliated open book} for a three-manifold with boundary $(M, \partial)$ consists of $(B, \pi)$, where $B$ is an oriented, properly embedded $1$-manifold and $\pi:M\setminus B\rightarrow S^1$ satisfies the following properties:
\begin{enumerate}
\item $\pi$ is an $S^1$-valued Morse function whose critical points all lie on $\partial M$;
\item $\pit:=\pi|_{\partial M}$ is Morse with the same set of critical points as $\pi$;
\item $\pi$ has a unique critical point for each critical value;
\end{enumerate}
\end{definition}

One of the first applications of classical Morse theory is to relate the topology of sublevel sets to the critical points of the Morse function.  Because the present function $\pi$ has only boundary critical points, the critical points instead detect changes in the topology of the level sets of $\pi$.  Equivalently, the critical points of $\pit$ detect changes in the topology of the level sets of $\pi$, and we record this  on $\partial M$.  Specifically, let $\mathcal{F}_\pit$ be the singular foliation whose leaves are the level sets of $\pit$, oriented as the boundary of the level sets of $\pi$, and each singular point comes with a sign: elliptic points are distinguished as positive sources and negative sinks, while (four-pronged) hyperbolic points are distinguished by the index of the critical point of $\pi$. An index two critical point of $\pi$ gives rise to a positive hyperbolic point of the singular foliation and corresponds to cutting a level set of $\pi$ along an arc.  An index one critical point of $\pi$ gives rise to a negative hyperbolic point of the singular foliation and corresponds to adding a one-handle to a level set of $\pi$.  

\begin{figure}[h!]
\begin{center}
\includegraphics[scale=0.7]{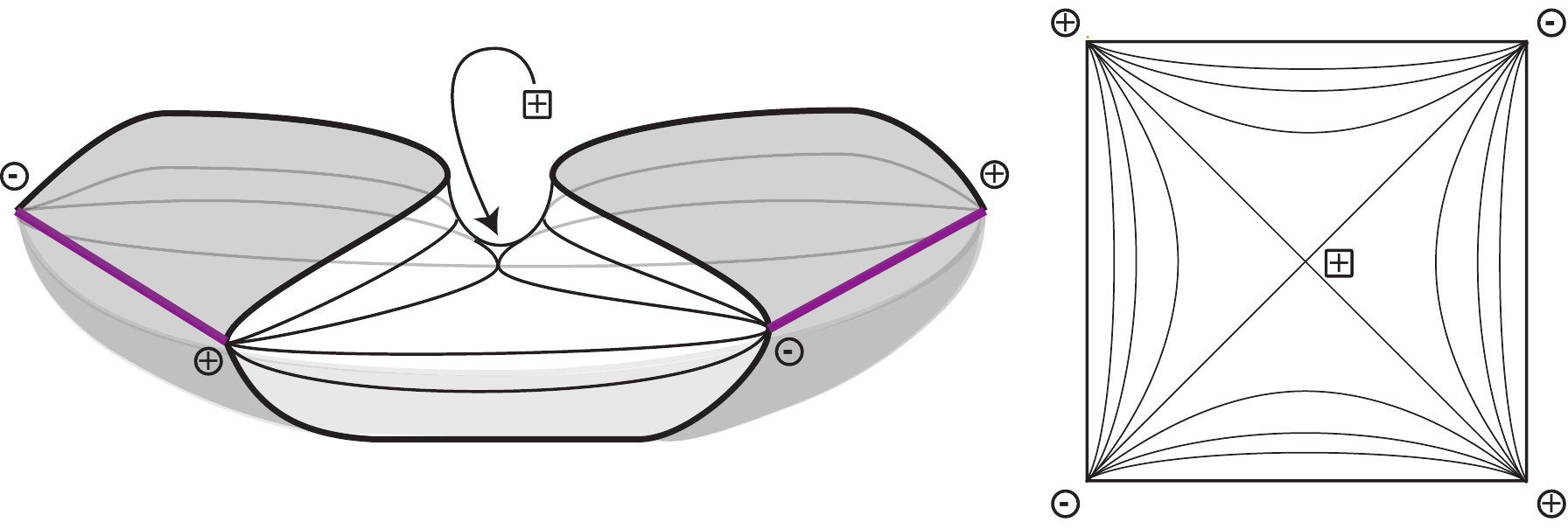}
\caption{ Left: a local cobordism between level sets near an index two boundary critical point of $\pi$.  Right: the boundary foliation $\mathcal{F}_\pit$ near the positive hyperbolic point. }\label{fig:levelsetfol}
\end{center}
\end{figure}

\begin{definition}  If $(B, \pi)$ is a foliated open book on $(M, \partial M)$, then $\mathcal{F}_\pit$ is the associated \emph{open book foliation} on $\partial M$.\end{definition}

By a slight abuse of the notation we call any (singular) foliation that comes from this construction an open book foliation, sometimes without a reference  to the enclosing open book and manifold. 

Open book foliations were first studied by Pavelsecu \cite{Pav} and by Ito-Kawamuro \cite{IK1} as a generalization of the braid foliations introduced by Birman-Menasco \cite{BM} that in turn rest on the ideas of Bennequin \cite{Ben}.  Our definition differs slightly from that of Ito-Kawamuro in the requirement that the foliation is by level sets of a Morse function, but this may be imposed in their setting without penalty.  

When $\mathcal{F}_\pit$ has no circle leaves, it admits a \textit{dividing set}, a curve $\Gamma$ defined up to isotopy as the boundary of a neighborhood of the positive separatrices from positive hyperbolic points.  If there is a (not necessarily smooth) isotopy taking one signed singular foliation to another through foliations sharing a dividing set, then we say that the two foliations are \textit{strongly topologically conjugate}.  This allows us to state the relationship between foliated open books and contact structures:

 \begin{definition}\label{def:support} \cite[Definition 3.7]{LV} The Morse foliated open book $(B, \pi)$ \emph{supports} the contact structure $\xi$ on $(M, \partial M)$ if $\xi$ is the kernel of some one-form $\alpha$ on $M$ satisfying the following properties:
 \begin{enumerate}
 \item $\alpha(TB)>0$;
 \item for all $t$, $d\alpha|_{\pi^{-1}(t)}$ is an area form; and
 \item $\mathcal{F}$ is strongly topologically conjugate to the characteristic foliation of $\xi$. 
 \end{enumerate}
 \end{definition} 
Note that condition (3) imposes that $\F_\pit$ has no circle leaves.

\begin{definition}\label{def:mxif} A \emph{foliated contact three-manifold} $(M, \xi, \mathcal{F})$ is a manifold with foliated boundary together with a contact structure $\xi$ on $M$ such that $\mathcal{F}$ is an open book foliation that is strongly topologically conjugate to $\mathcal{F}_{\xi}$.  \end{definition} 

\begin{theorem}\cite[Theorems 3.10, 6.9, 7.1, 7.2]{LV} \label{thm:support}
Any foliated open book supports a unique isotopy class of contact structures, and any foliated contact three-manifold $(M,\xi,\mathcal{F})$ admits a supporting foliated open book.  Two foliated open books for the same foliated contact three-manifold are related by positive stabilization.
\end{theorem}

Although we have not yet defined positive stabilization in this article, we note that it is an internal operation analogous to stabilization on other forms of open books.

\subsection{Preferred gradient-like vector fields and the monodromy}
As in the case of standard Morse theory, the benefit of a Morse function is fully realised only in the presence of a gradient-like vector field.  We will designate a  class of gradient-like vector fields as \textit{preferred} based on their compatibility with the open book foliation on $\partial M$.  Suppose that $\mathcal{F}_\pit$ is the open book foliation on a manifold $M$ supporting some contact structure with convex boundary.  As $\mathcal{F}_\pit$ has no circle leaves, it decomposes the surface as a union of square tiles, each of which contains a single hyperbolic singularity in the middle, four elliptic singularities at the corners, and four connected components of leaves as edges.   In fact, one may recover the entire open book foliation up to  topological equivalence by decomposing a surface into squares labeled with signs, signed corners, and an order in which the hyperbolic singularities appear.  See Figure~\ref{fig:grad} for an illustration of a single tile and Figure~\ref{fig:otnbhd} for an example of a surface decomposed into squares.

\begin{definition}  A gradient-like vector field $\nabla \pi$ is \emph{preferred} if it is tangent to $\partial M$ on $\partial M$ and if the flowlines on each tile of the foliated surface $\partial M$ are isotopic to those shown in Figure~\ref{fig:grad}.  

\begin{figure}[h!]
\begin{center}
\includegraphics[scale=0.8]{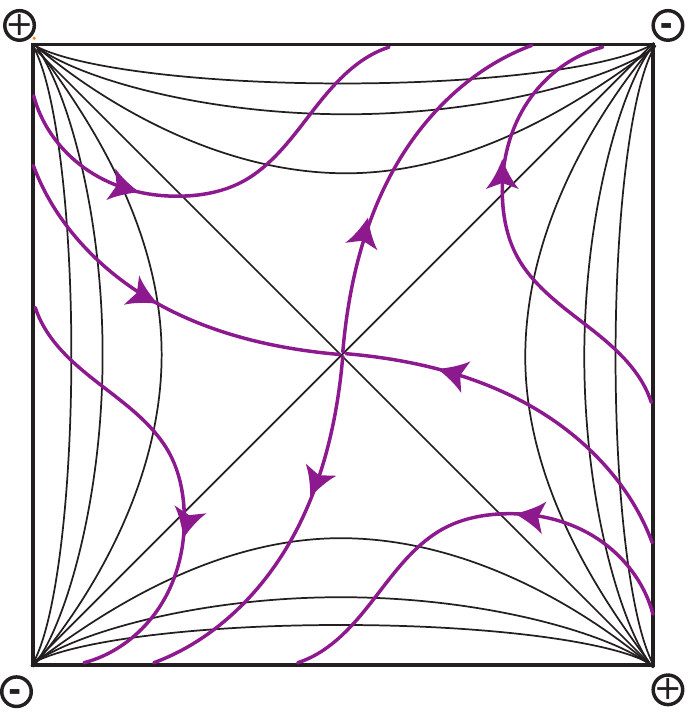}
\caption{ Flowlines of a preferred gradient-like vector field on the tile defined by a hyperbolic point.  The hyperbolic point could be positive or negative. }\label{fig:grad}
\end{center}
\end{figure}
\end{definition}

We will assume henceforth that $\nabla \pi$ is always preferred.

Said slightly differently, $\nabla \pi$ is preferred if it is the extension to $M$ of some $\nabla \pit$ with presecribed properties on $\partial M$.   The condition on flowlines is chosen so that topological properties of the original foliation are reflected in topological properties of the flowlines of $\nabla \pit$.  Specifically, we may consider the graph formed by the positive separatrices of positive hyperbolic points, just as above in the definition of $\Gamma$. %associated to each positive elliptic point which consists of the positive separatrices of positive hyperbolic points that terminate at $e^+$.   
The flowlines of the gradient-like vector field spiral around the elliptic points infinitely many times, but we may nevertheless define a similar graph: replace each positive elliptic point by the connected component of $\pit^{-1}(0)$ which terminates at it, and let the graph be this component together with the union of positive flowlines of $\nabla \pit$ from positive hyperbolic points truncated when they hit these intervals first.  After a deformation retract of the leaf to its positive endpoint, these two graphs are isotopic.  The analogous statement holds for the analogously defined graph constructed from negative separatrices of negative elliptic points.  

\begin{figure}[h!]
\begin{center}
\includegraphics[scale=0.8]{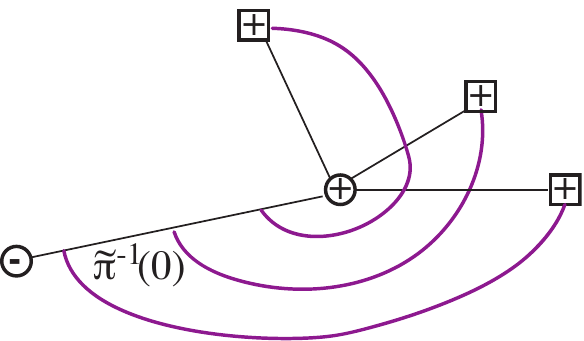}
\caption{  The squares represent positive hyperbolic points with a positive separatrix terminating at the shown positive elliptic point (circle).  The separatrices of $\nabla \pit$ intersect $\pit^{-1}(0)$ in the same order that the separatrices of $\mathcal{F}$ intersect the elliptic point. }\label{fig:star}
\end{center}
\end{figure}

We now assume every foliated open book is equipped with a preferred gradient-like vector field, and we write $(B, \pi, \mathcal{F}_\pit, \nabla\pi)$ to denote the additional data described above.  Choosing a gradient-like vector field has the immediate consequence of determining  critical submanifolds associated to each hyperbolic point.   In fact, we will truncate each critical submanifold at its first intersection with $\pi^{-1}(0)$. If $h^+$ is a positive hyperbolic point of $\mathcal{F}_\pit$ (i.e., an index-two critical point of $\pi$) with $\pi(h^+)=c_+$, its stable submanifold  intersects each level set of $\pi^{-1}[0, c_+]$.  Denote these intersections by $\gamma^+$, specifying the level set containing $\gamma^+$ if necessary.  Similarly, the unstable critical submanifold of a negative hyperbolic point $h^-$ (i.e., an index-one critical point of $\pi$) with $\pi(h^-)=c_-$ intersects each level set of $\pi^{-1}[c_-, 1]$ and we denote each of these intersections by $\gamma^-$. Observe, as in Figure~\ref{fig:submfd}, that if the unstable and stable critical submanifolds intersect, then  a single critical submanifold may be represented on subsequent pages by multiple arcs.

\begin{figure}[h!]
\begin{center}
\includegraphics[scale=0.6]{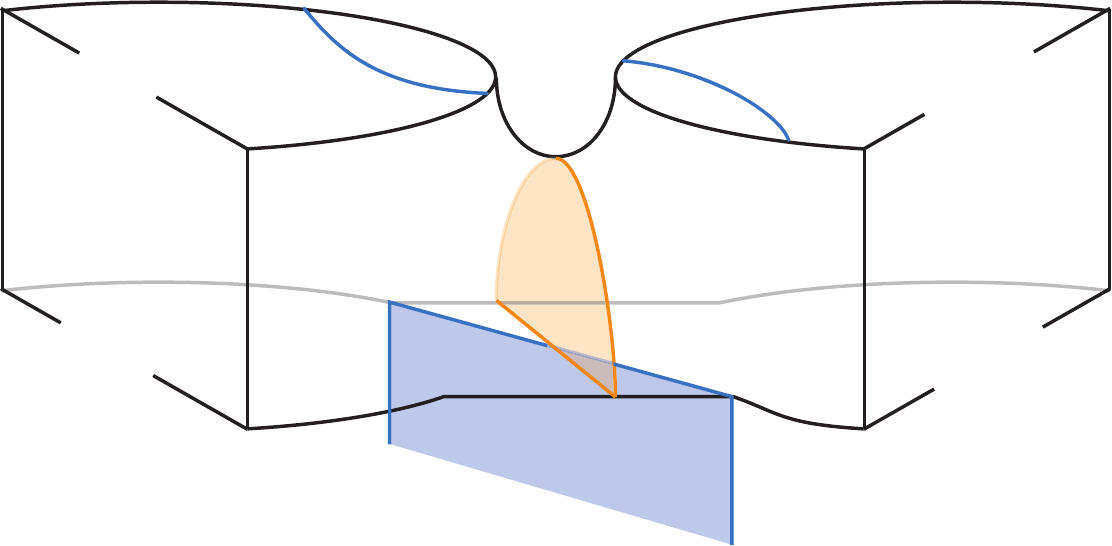}
\caption{ The orange descending critical submanifold intersects the blue ascending critical submanifold, cutting the intersection curve on all subsequent pages.}\label{fig:submfd}
\end{center}
\end{figure}

\begin{example}\label{ex:torusnotsorted} As a first example, we consider a Morse foliated open book for a solid torus.  Begin with an embedding of the solid torus in $S^3$ as shown in Figure~\ref{fig:notsorted}.  As proven in \cite{LV}, the function defined as the radial coordinate of the embedding  may be perturbed near the boundary of the solid torus so that the resulting restriction defines a Morse foliated open book with the same critical points on the boundary.   There are four critical points, alternating in sign, and  $S_1, S_2,$ and $S_3$ in Figure~\ref{fig:notsorted} are pages ---that is, level sets--- for regular values separated by critical points. Pages $S_0$ and $S_4$ are not separated by a critical level, but we include both in order to see the intersections of both the ascending (on $S_4)$ and descending (on $S_0$) critical submanifolds. 

\begin{figure}[h!]
\begin{center}
\includegraphics[scale=0.8]{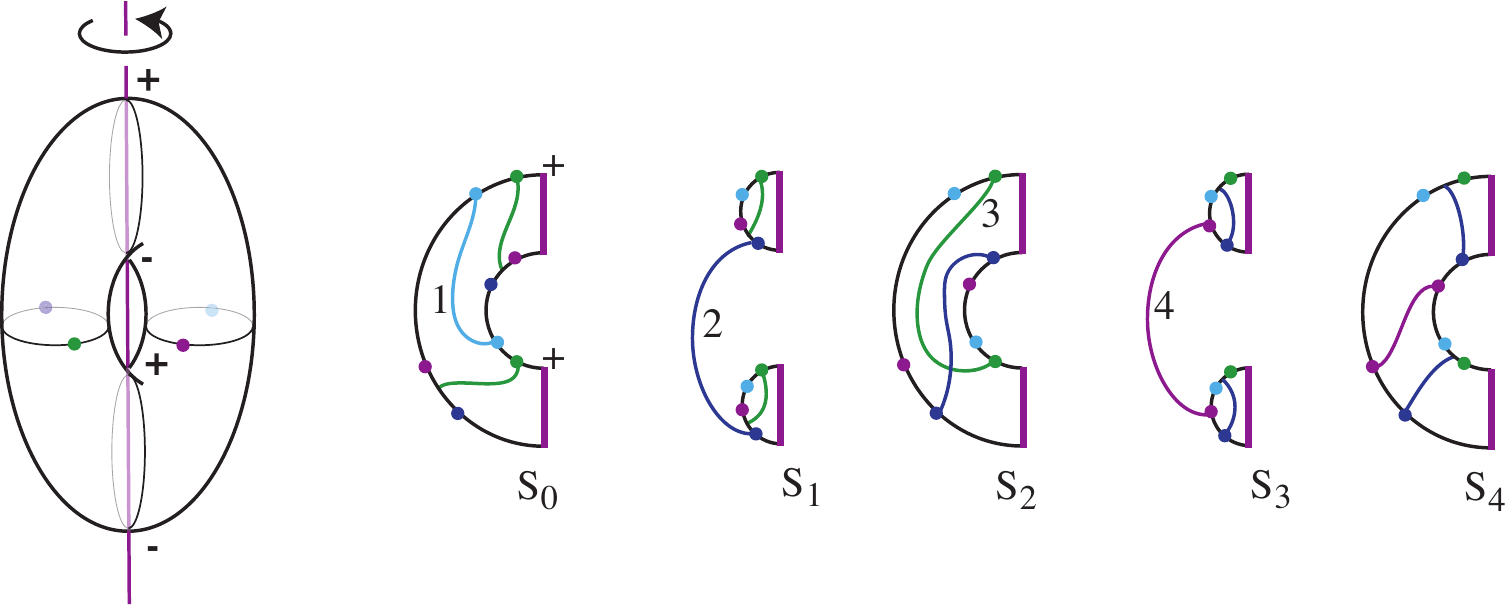}
\caption{ Level sets $S_q:=\pi^{-1}(\frac{q\pi}{2}+\epsilon)$ for a Morse foliated open book on a solid torus, together with their intersections with the truncated critical submanifolds.  The bold dots along the non-binding boundary of the pages indicate the relative positions of the intersection arcs $\gamma^\pm$, a consequence of choosing $\nabla \pi$ to be preferred.}\label{fig:notsorted}
\end{center}
\end{figure}
\end{example}

\begin{lemma} The  truncated stable critical submanifolds are disjoint.  Similarly, the truncated unstable critical submanifolds are disjoint.
\end{lemma}

\begin{proof} The lemma follows from the fact that $\nabla \pit$ is Morse-Smale.
\end{proof}

We now consider the flowline through an arbitrary point $p$ in the interior of $\pi^{-1}(0)$.  If $p\in \gamma_i^+$ for some positive hyperbolic point $h_i^+$, then the flowline through $p$ will terminate at $h_i^+$.  However, for all points in the complement of the $\{\gamma_i^+\}$, there is a well-defined first-return map.  Define $P:=\pi^{-1}(0)\setminus (\cup_i \gamma_i^+)$.

\begin{definition} The \emph{monodromy} $H: P \rightarrow P$ of a foliated open book is the first return map of $\nabla \pi$.  
\end{definition}

\begin{remark}This use of the term monodromy differs slightly from the map $h$ called the monodromy in \cite{LV}, \cite{AFHLPV}.   The domain of $H$ is a proper subset of the $S_0$ page, while $h$ is a homeomorphism from the final page to the initial on.  On $P$, $H=h\circ \iota$.
\end{remark}

\subsection{From foliated to partial open books} Foliated and partial open books are alternative ways to decompose contact manifolds with boundary, and unsurprisingly, they are closely related. To each foliated open book we associate a triple $(S, P, H)$, where $S:=\pi^{-1}(0)$, and $P$ and $H$ are as above.  Although $S$ and $P$ may be viewed as abstract surfaces, rather than embedded ones, we retain the identifications  $P\subset S$ and $\partial S=B \cup \pit^{-1}(0)$, where $B=\partial S\cap \partial P$.  Under certain circumstances described below, the triple associated to a foliated open book in fact defines a partial open book for the same contact manifold.  

\begin{remark} In \cite{LV},  $P$ is defined by removing a neighborhood of the arcs $\gamma_i^+$ together with the non-binding boundary of $S$.  However, the resulting subsurface is isotopic to the $P$ defined above and the results stated for triples are all independent of this choice.
\end{remark}

\begin{definition} A Morse foliated open book $(B, \pi, \mathcal{F}_\pit, \nabla \pi)$ is \emph{sorted} if there are no flowlines contained in $\pi^{-1}(0,1)$ between distinct critical points.\end{definition} 

Equivalently, a foliated open book is sorted if the set of all truncated critical submanifolds is disjoint.  We note that this definition highlights the role of preferred gradient-like vector fields, as the constraints on the boundary will force intersections that could otherwise be avoided.

\begin{proposition}[ref in \cite{LV}]
Suppose that $(S, P, H)$ is the triple associated to a sorted foliated open book.  If $S$ may be built up from $S\setminus P$ by attaching one-handles along $0$-spheres embedded in $P$, then the triple $(S, P, H)$ defines a partial open book for $(M, \partial M, \Gamma)$. 
\end{proposition}

As in the case of other forms of open books, foliated open books admit an operation called \emph{positive stabilization} which preserves the supported contact structure.  Positive stabilization is equivalent to taking an appropriate connect sum with a foliated open book for the standard tight $S^3$.  A positive stabilization is determined up to equivalence by a choice of properly embedded arc  $\gamma\subset \pi^{-1}(t)$ with $\partial \gamma\subset B$, and it changes each level set of $\pi$ by attaching a $1$-handle along $\partial \gamma$ which becomes part of $P$.  The monodromy of the foliated open book changes by a positive Dehn twist along the circle formed by $\gamma$ and core of the added $1$-handle, restricted to $P$.%{probably completely unclear, not sure how to say this well}

%%%%%%%%%%%%%%%%%%%%%%%%%%%%%%%%%%%%%%%%%%%%%%%%%%%%%%%

% !TEX root = RV_main.tex
%%%%%%%%%%%%%%%%%%%%%%%%%%%%%%%%%%%%%%%%%%%%%%%%%%%%%%%
\section{Right-veering monodromies and examples}\label{sec:rv}

\begin{definition} Let $\gamma, \delta$ be propertly embedded arcs with $\partial \gamma=\partial \delta$.  We write $\delta < \gamma$ if, after isotoping the two arcs relative to their shared boundary so that they intersect minimally, $\delta$ does not lie to the left of $\gamma$ near either endpoint. If $\delta<\gamma$ or $\delta$ is isotopic to $\gamma$, then we write $\delta \leq \gamma$.%actually, didn't check this definition or use of $<$ 

\end{definition}

Consider a surface $S$ and a subsurface $P\subset S$. Let $H:P\rightarrow S$ be an embedding  which restricts to the identity on $\partial P \cap \partial S$.  

\begin{definition} Given $(S, P, H)$ as above, $H:P\rightarrow S$ is \emph{right-veering} if for every $\gamma$ properly embedded in $P$ with $\partial \gamma \subset (\partial P\cap \partial S)$, $H(\gamma) \leq \gamma$. 
\end{definition}

\begin{definition} The  foliated open book $(B, \pi, \mathcal{F}_\pit, \nabla \pi)$ is \emph{right-veering} if the associated triple $(S, P, H)$ is right-veering.
\end{definition}

We note that this definition does not depend on the choice of preferred $\nabla\pi$.  Any two preferred gradient-like vector fields for a fixed $\pi$ agree near $\partial M$, and they are connected by a path of preferred gradient-like vector fields all fixed near the boundary.  In particular, this implies that the flowlines on the boundary are preserved throughout the interpolation.  It follows that  the arcs $\gamma_{i}^+$ on $S$ may change only by  isotopy, as an arc slide would required the Morse-Smale condition to fail at some point.

When the triple $(S, P, H)$ defines a partial open book, the previous two definitions exactly coincide with those of Honda-Kazez-Matic.  However, we observe that the definitions here are broader in scope, applying to the triple $(S, P, H)$ associated  to an arbitrary Morse foliated open book.  

\begin{example}Consider the solid torus of Example~\ref{ex:torusnotsorted}. As seen in Figure~\ref{fig:notsorted}, $P$ consists of a union of four discs.  It follows that the monodromy $H$ restricted to each component is trivial, so the Morse foliated open book is necessarily right-veering. Note, however, that this triple does not define a partial open book; since some components of $P$ meet $S\setminus P$ along a single curve, $S$ cannot be built up from $S\setminus P$ by one-handle addition.
\end{example}  

We next show an advantage of extending these definitions; informally, it allows one to recognize left-veering monodromy (and hence, overtwistedness) in simpler objects. The following example begins with a left-veering Morse foliated open book that does not define a partial open book.  However, after performing a single stabilization, the resulting Morse foliated open book defines a right-veering partial open book.

\begin{example}\label{ex:stabtorv}

We describe $(B, \pi, \mathcal{F}_\pit, \nabla \pi)$ via three of its regular pages; in fact, this determines  several distinct foliated open books depending on the order in which the positive (respectively, negative) hyperbolic points appear in the foliation, but the example does not depend on this choice.  Suppose that the map  $h: \pi^{-1}(1-\epsilon)\rightarrow \pi^{-1}(\epsilon)$ is a left-handed Dehn twist relative to the page, so that the dotted arc in Figure~\ref{fig:lvfob} is evidently left-veering.  Note, too, that the associated triple does not define a partial open book, as the bigon components of $P$ cannot be built up from the rest of the page by attaching one-handles.
\begin{figure}[h!]
\begin{center}
\includegraphics[scale=.8]{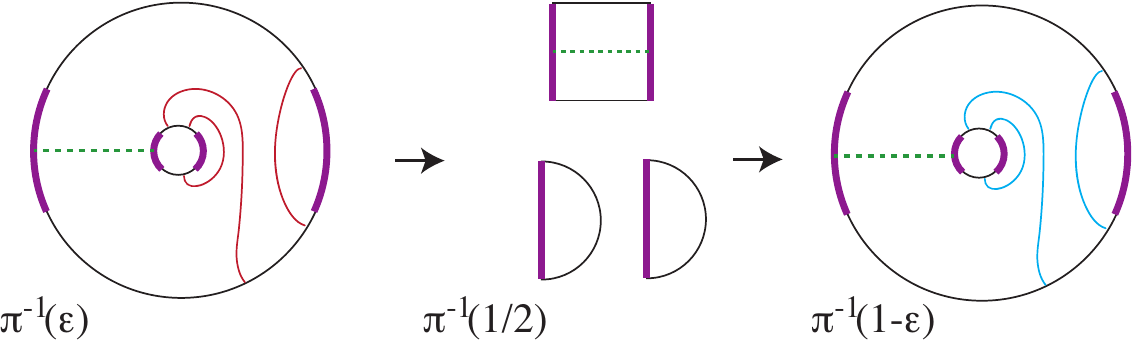}
\caption{   Selected regular pages decorated with their intersections with the truncated critical submanifolds.}\label{fig:lvfob}
\end{center}
\end{figure}
 
 Now stabilize the foliated open book along $\gamma\subset \pi^{-1}(1-2\epsilon)$ of Figure \ref{fig:rvpob}, adding a one-handle to $P$ on each page and changing the monodromy by a positive Dehn twist on $\pi^{-1}(1-2\epsilon)$.  After the dotted arc flows through this page, it is disjoint from the core of the original annulus, so it remains unaffected by the negative Dehn twist.  Since $P$ consists of just two rectangular components, it is easy to verify that there are no left-veering arcs.  
 
 \begin{figure}[h!]
\begin{center}
\includegraphics[scale=.8]{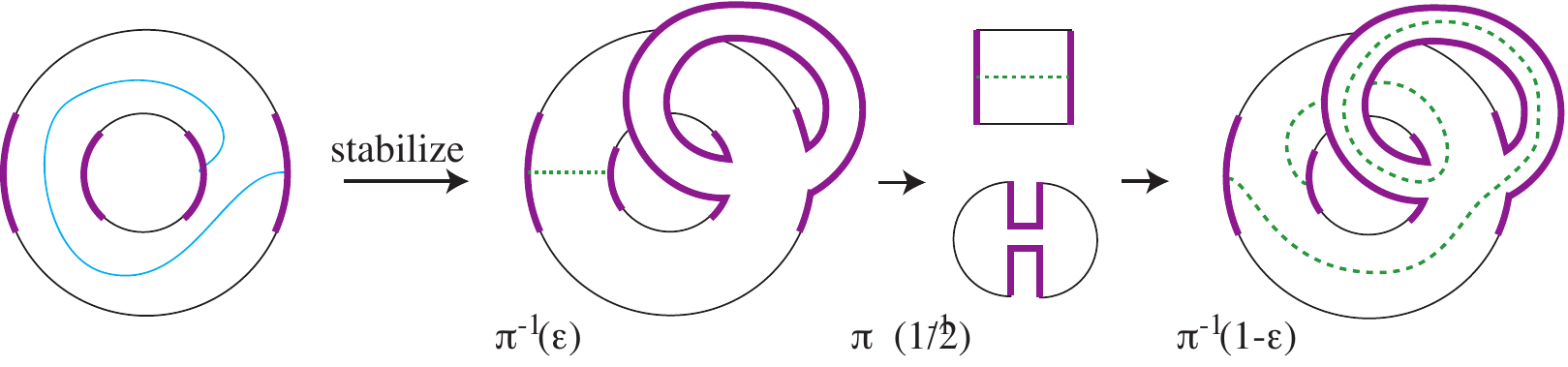}
\caption{ After stabilizing along the indicated arc, the flow changes by a positive Dehn twist. }\label{fig:rvpob}
\end{center}
\end{figure}

The appeal of this example lies in the fact that it detects left-veering behavior ---and hence, overtwistedness--- in a simple object, a foliated open book whose associated triple satisfies weaker conditions than those required by a partial open book.  However, it also highlights the difference between equivalence classes of contact manifolds with foliated boundary and those with merely convex boundary.  Honda-Kazez-Matic have shown that every overtwisted contact manifold with convex boundary is supported by some left-veering open book; in fact, the partial open book of Figure~\ref{fig:rvpob} is a stabilization of a left-veering open book, but this stabilization changes $|B|$, and hence does not preserve the foliated boundary.  %need HKM reference(s) here

\begin{figure}[h!]
\begin{center}
\includegraphics[scale=1.0]{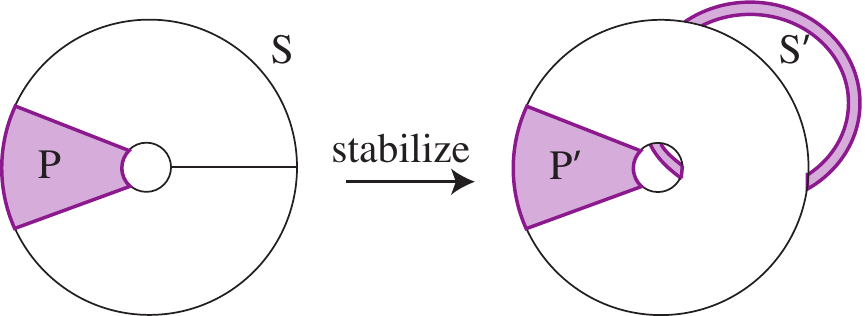}
\caption{This partial open book stabilization preserves the equivalence class of the convex boundary, but not the foliated boundary. }\label{fig:lvpob}
\end{center}
\end{figure}

\end{example}

Nevertheless, a left-veering arc in a foliated open book does  play the same role as in other forms of open books:

\begin{theorem}\label{thm:ot} 
A contact three-manifold with foliated boundary is overtwisted if and only if it is supported by a foliated open book that is not right-veering.
%Let $(M,\xi,\Gamma)$, $(M',\xi',\Gamma')$ be a contact three-manifolds with convex boundary. %and fix a open book foliation $\mathcal{F}$ on $\partial M$ that is divided by $\Gamma$. 
%Then the following are equivalent:
%\begin{enumerate}
%\item\label{it:1} $(M,\xi,\Gamma)$ is tight;
%\item\label{it:2} $(M,\xi,\mathcal{F})$ is tight for some open book foliation $\mathcal{F}$ on $\partial M$ that is divided by $\Gamma$;
%\item\label{it:3} $(M,\xi,\mathcal{F})$ is tight for all open book foliation $\mathcal{F}$ on $\partial M$ that are divided by $\Gamma$;
%\item\label{it:4} all partial open books for  $(M,\xi,\Gamma)$ are right veering;
%\item\label{it:5} for some open book foliation $\mathcal{F}$ on $\partial M$ divided by $\Gamma$, all foliated open books for $(M,\xi,\mathcal{F})$ are right-veering;
%\item\label{it:6} for all open book foliations $\mathcal{F}$ on $\partial M$ divided by $\Gamma$, all foliated open books for $(M,\xi,\mathcal{F})$ are right-veering.
%\end{enumerate}
\end{theorem}
%\begin{remark}
%In particular this means, that none of the above depends on the choice of $\F$.
%\end{remark}

%\vv{\begin{theorem}\label{thm:ot} The contact 3-manifold with foliated boundary $(M, \xi, \mathcal{F})$ is overtwisted if and only if  it admits a foliated open book $(B, \pi, \mathcal{F}_\pit, \nabla\pi)$  which not right-veering.\end{theorem}}

\begin{remark}
Observe that the overtwistedness of a contact 3-manifold with foliated boundary does not depend on the particular foliation on the boundary, but pnly on the dividing set $\Gamma$ associated to the foliation. Thus by the above theorem, if two contact structures differ only near  $\partial M\times I$, and there they are both foliated by convex surfaces $\partial M\times \{t\}$, then one of them is supported by a non-left veering foliated open book if and only if the other is. 
\end{remark}
The ``if'' direction of Theorem \ref{thm:ot} follows from the following result:

\begin{proposition}\label{prop:otfob} 
If a foliated open book is left-veering, then the supported contact structure is overtwisted.
\end{proposition}

\begin{proof} 
We prove this result via the analogous statement for partial open books.  Specifically, we will show that if a foliated open book has a left-veering arc, then we may construct a partial open book for the same contact manifold that also has a left-veering arc.  By the work of Honda-Kazez-Matic, this implies the supported contact structure is overtwisted.  If the triple associated to the foliated open book already defines a partial open book, there is nothing to do, so we consider the following case.

Suppose that $(B, \pi, \mathcal{F}_\pit, \nabla \pi)$ has a left-veering arc but the associated triple $(S, P, H)$ does not define a partial open book. It is always possible to stabilize the foliated open book so that the triple defines a partial open book, and we show that these stabilizations may be chosen to preserve the left-veering arc. 

A stabilization is completly determined by a choice of stabilizing arc on a page.  We show that under the conditions of the proposition, a sequence of stabilizing arcs may be chosen  that are both disjoint from the fixed left-veering arc $\gamma\subset P$ and which produce a partial open book as the associated triple.

Note first that there are two ways in which the triple associated to a Morse foliated open book may fail to define a partial open book.  First, the Morse foliated open book may not be sorted, and  second, if may be impossible to build $S$ up from $S\setminus P$ by one-handle addition.  

If a foliated open book is not sorted, \cite{LV} describes how to choose stabilizing arcs to produce a sorted version.  Starting from $t=0$, increase the $t$ value until the first time when a flowline between a pair of critical points appears.  The stabilization should be performed on a page intersecting this flowline, and except for at its endpoints, the arc should be chosen to lie in a neighborhood of the non-binding boundary and the intersections between critical submanifolds and the chosen page, $\pi^{-1}(t_0)$.  A left-veering arc on $\pi^{-1}(0)$ lies completely in $P$, and by definition, it will flow to an arc in $\pi^{-1}(t_0)$ which is disjoint from the critical submanifolds and the non-binding boundary of the page.  
The tstabilizing arcs involved in rendering a foliated open book sorted may be assumed disjoint from any left-veering arcs.  

If $(S, P, H)$ fails to define a partial open book because $S$ cannot be built up from $S\setminus P$ by one-handle addition, then there are components of $P$ which intersect $S \setminus P$ along a single interval in $\partial P$.  Fix one such component, and choose a next component with respect the the boundary orientation around $S$.  Choose a boundary parallel stabilizing arc that connects the two components; this may clearly be done in the complement of a left-veering arc that is properly embedded in $P$.  This argument misses the case when $P$ consists of single component; however, since $P$ is cut from $S$ by a two-sided arc, such a subsurface meets $S\setminus P$ in two components. 

%Suppose first that the stabilizing arc $\gamma$ lies completely in a fixed component of $P$ on some page.   Stabilization along an arc entirely contained in $P$ will not remove intersections between stable and unstable critical submanifolds.  (See Section 8.3 of \cite{LV}.) On the other hand, if a component of $P$ meets $S\setminus P$ along only one boundary interval, stabilizing along an arc contained in $P$ preserves the attaching intervals between $P$ and $S\setminus P$.  Thus the resulting triple will again fail to define a partial open book.   

%Now suppose that the stabilizing arc lies partially outside $P$.  The only new essential arc in $P'$ is the cocore of the added handle, and since $h'$ acts by a positive Dehn twist on the union of the handle core and the stabilizing arc...
\end{proof}

%\jel{I actually rewrote 3.8 as we'd discussed, so I think this paragraph is extraneous.  If you'd prefer to have a weaker statement above, that's okay with me.}
%\vv{We start with the ''if`` direction. If the contact 3-manifold with foliated boundarty $(M,\xi,\F)$ has a non right veering foliated open book then by Proposition \ref{prop:otfob} we can construct a non right veering partial open book for the corresponding contact structure $(M,\xi,\Gamma)$, where $\Gamma$ is the dividing set for $\F$. By  \cite[Proposition 4.1.]{hkm09} the contact structure  $(M,\xi,\Gamma)$ is overtwisted, and thus $(M,\xi,\F)$ is overtwisted as well. }

The proof of the ''only if`` direction in Theorem~\ref{thm:ot} is an amalgamation of the constructions of \cite[Section 8.5 ]{LV} and \cite[Proposition 4.1.]{hkm09}.  Although these  require the use of embedded foliated open books, rather than the Morse foliated open books highlighted in this article,  the statement is such a close fit to the topic that we have elected to include it anyway.

\begin{proof}[Proof of Theorem \ref{thm:ot}:]
Section 8.5 of  \cite{LV}  proves the existence of a supporting foliated open book for a given contact manifold via a partial open book adapted to the given foliation near the boundary. The proof uses the foliation to construct this adapted partial open book near the boundary and then extends it  into the interior of the manifold in a standard way using a contact cell decomposition. On the other hand,  the proof of  \cite[Proposition 4.1.]{hkm09} considers a (non-right-veering) partial open book for a neighborhood of an overtwisted disc, connects it with a Legendrian arc to the portion of the partial open book that is constructed near the boundary, and then extends it in a standard way using a contact cell decomposition for the complement. 

We now prove that an overtwisted contact three-manifold with foliated boundary $(M,\xi,\F)$ has a non-left-veering foliated open book.  Take a partial open book for a neighbourhood of an overtwisted disc, connect it with a Legendrian arc to the portion of partial open book near the boundary that is adapted to the foliation, and then extend it in a standard way using a contact cell decomposition for the complement. This partial open book now can be extended to a foliated open book that supports $(M,\xi,\F)$ and the non-right-veering arc is preserved throughout the extension process.
\end{proof}

\begin{corollary} If $(B, \pi, \mathcal{F}_\pit, \nabla \pi)$ is left-veering, then the bordered sutured contact invariant $c(\xi)$ vanishes.
\end{corollary}

\begin{proof} Theorem 3 in \cite{AFHLPV} shows that under a certain isomorphism between the bordered sutured Floer homology of a foliated open book and the sutured Floer homology associated to the corresponding partial open book,  the bordered sutured contact invariant $c(\xi)$ maps to the Honda-Kazez-Matic invariant $EH( M, \Gamma, \xi)$.  Since the $EH$ invariant vanishes for overtwisted contact manifolds, the result follows.
  \end{proof}
%(Careful: does any partial open book necessarily define a sorted foliated open book?  If this also requires stabilization, we should indicate this.)

\begin{example}\label{ex:tot}   Here, we describe a Morse foliated open book for an overtwisted three-ball and then show how the existence of a left-veering arc guides the construction of a transverse overtwisted disc. This is essentially the construction described in \cite{IKtot}, but some adaptation is required because the topological type of the page changes with the $S^1$ parameter. 

Figure~\ref{fig:otnbhd} shows an $S^2$ decomposed into ten squares, each of which represents the square tile in the boundary open book foliation defined by a single hyperbolic point.  The integers indicate the order of the hyperbolic points, so that the first is a saddle resolution which transforms AJ and LE leaves into AE and LJ leaves.  The first five hyperbolic points are all positive, so each of the first five changes to the topology of the page is given by cutting along a single arc where the corresponding critical submanifold intersects the page.  These arcs are shown together on the first page in Figure~\ref{fig:otnbhdpages}, showing that $P$ consists of five discs.  There is a unique non-boundary-parallel arc, shown as a dotted curne, in the disc with corners labeled GLKH.

\begin{figure}[h!]
\begin{center}
\includegraphics[scale=1.0]{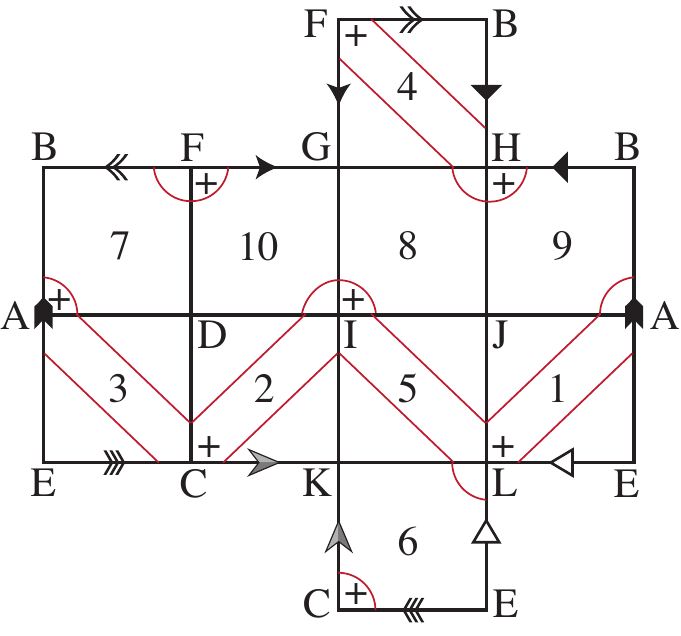}
\caption{Hyperbolic tiles of the open book foliation on the boundary of an overtwisted three-ball. The dividing set $\Gamma$ is shown in red. }\label{fig:otnbhd}
\end{center}
\end{figure}

It is not difficult to show that the manifold defined by these pages is a ball; to see that it's an overtwisted ball, recall from Definition~\ref{def:support} that the open book foliation is topologically conjugate to the characteristic foliation of a supported contact structure.  Thus, the dividing set on $S^2$ is the boundary of a neighborhood of the positive separatrices of the positive hyperbolic points.  Shown in red on Figure~\ref{fig:otnbhd}, the resulting $\Gamma$ has three components.

\begin{figure}[h!]
\begin{center}
\includegraphics[scale=1.0]{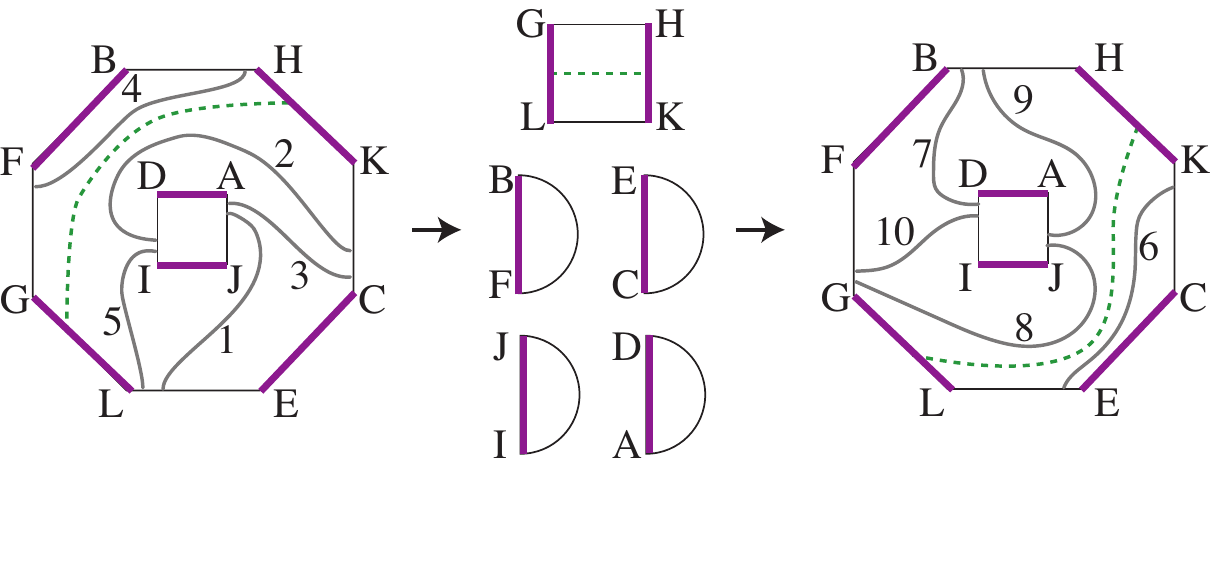}
\caption{Selected pages of the Morse foliated open book decorated with their intersections with the truncated critical submanifolds, each labeled by the corresponding hyperbolic point.  The dotted arc is left-veering under the foliated open book monodromy that identifes the right-hand page with the left-hand page by translation.   }\label{fig:otnbhdpages}
\end{center}
\end{figure}

 Following the approach of \cite{IKtot}, we construct a transverse overtwisted disc by describing how it intersects each page.  As $t$ changes, this intersection changes either by isotopy or by a saddle resolution, as shown in Figure~\ref{fig:otnbhddisc}.
 
\begin{figure}[h!]
\begin{center}
\includegraphics[scale=0.8]{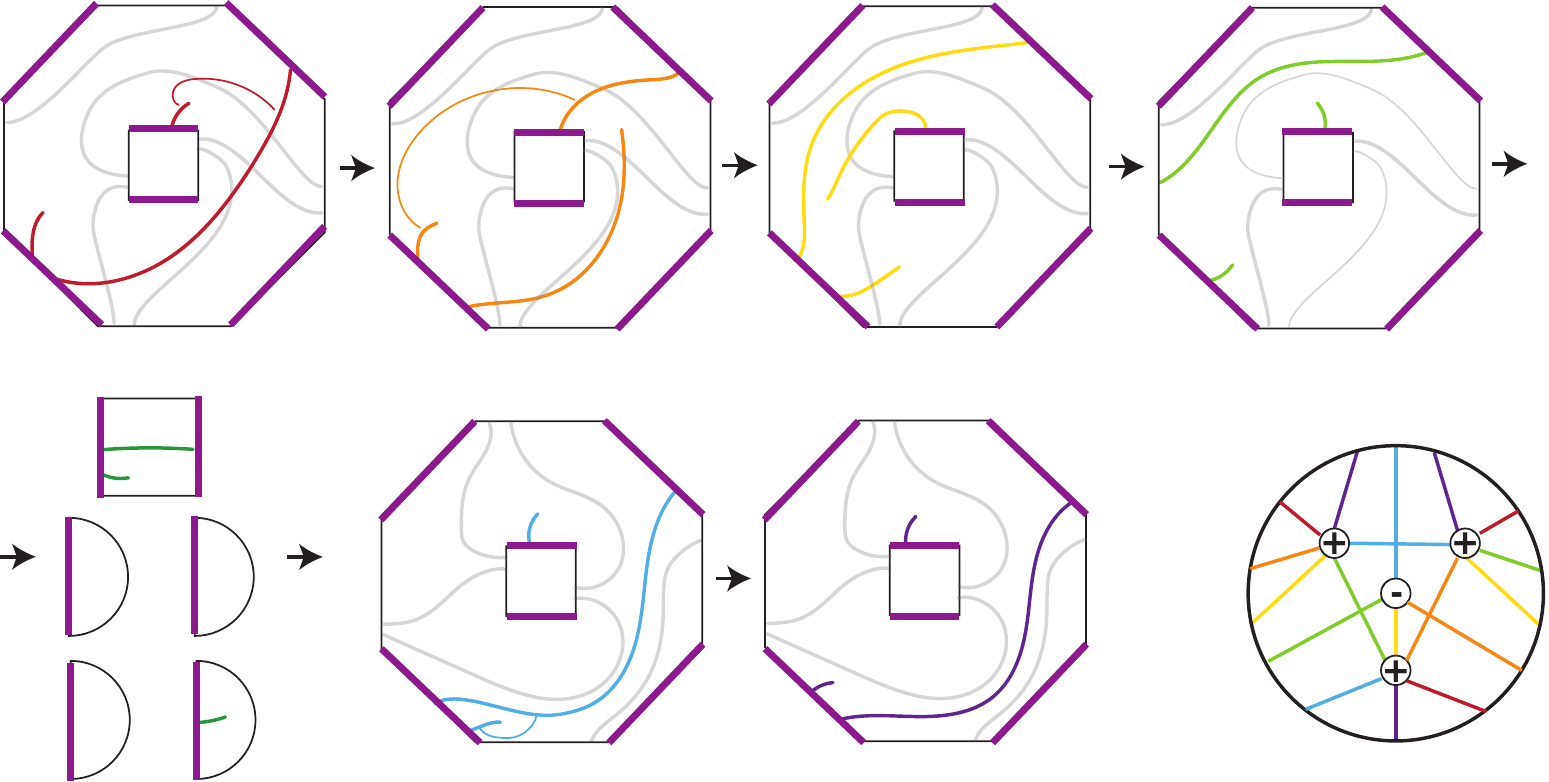}
\caption{Movie presentation of the tranverse overtwisted disc shown on the lower right.  The steps between successive slices are given by isotopy of the bold arcs and saddle resolutions guided by the thin arcs.}\label{fig:otnbhddisc}
\end{center}
\end{figure}

\end{example}

Although the existence of a left-veering arc allows us to construct an overtwisted disc, certainly there are foliated open books for overtwisted contact manifolds that don't have left-veering arcs.  There are several constructions in the literature that show how a left-veering open book may become a right-veering open book via a sequence of positive stabilizations  but the next example illustrates that even an unstabilized foliated open book for an overtwisted contact manifold may fail to be left-veering. \cite{HKMrv1}

\begin{example}\label{ex:minot}
Finally, we turn to a ``minimal" neighborhood of an overtwisted disc.  Beginning with the simplest open book foliation on a transverse overtwisted disc, we first construct something that is almost a foliated book by thickening the disc and taken the thickened leaves of the foliation as pages.  (Here, the function to $S^1$ is induced from the $S^1$ function on the foliation.)  However, each of the two critical pages constructed thus has a pair of critical points, so we locally perturb the Morse function so that $\pi(h_1)<\pi(h_2)<\pi(h_3)<\pi(h_4)$.  This yields the regular pages shown in Figure~\ref{fig:minotnbhd}.  Each page is decorated with its intersections with the ascending critical submanifold of hyperbolic points with lower values and the descending critical submanifold of hyperbolic points with greater values.

\begin{figure}[h!]
\begin{center}
\includegraphics[scale=0.8]{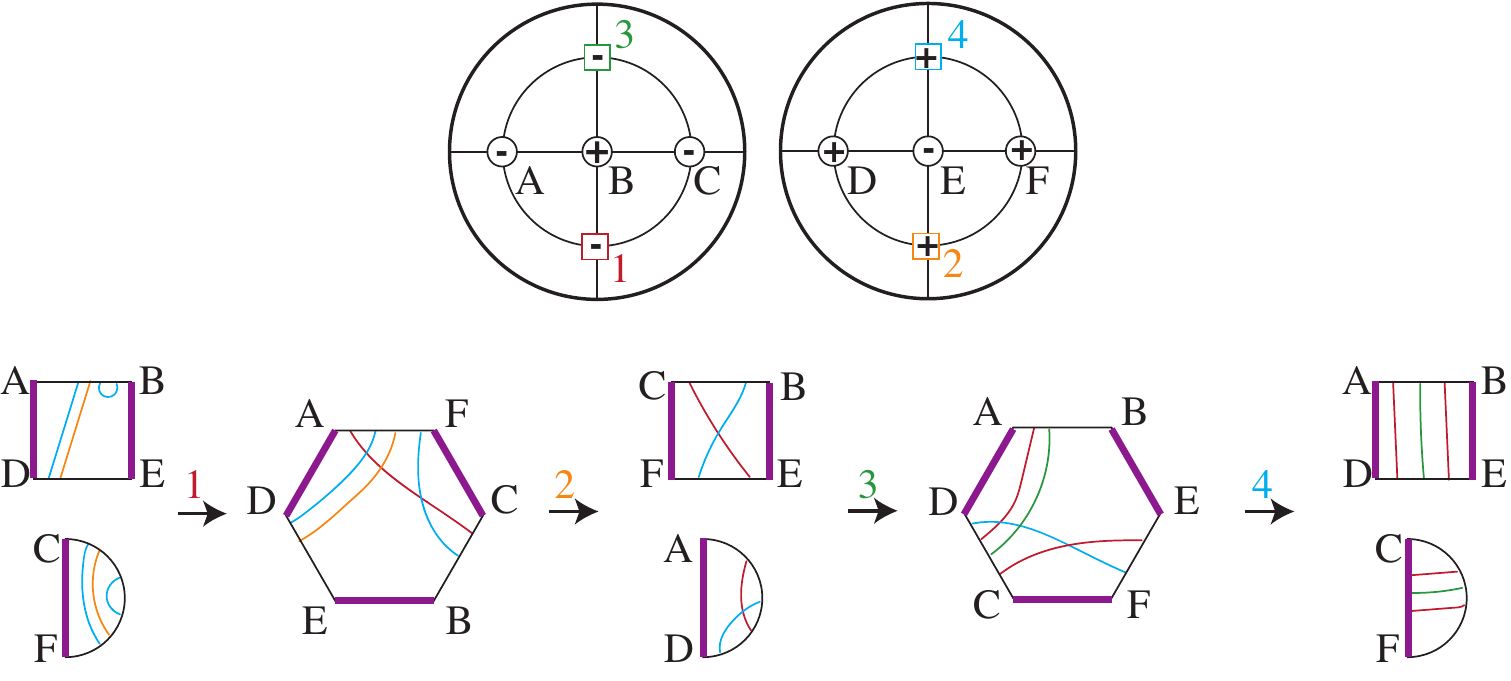}
\caption{ Top: selected leaves of the open book foliation on the front (left) and reverse (right) of a transverse overtwisted disc.  Circles are elliptic points and squares are hyperbolic points.  Bottom:  Regular pages of a Morse foliated open book for a neighborhood of the disc, decorated with the their intersections with the truncated critical submanifolds. }\label{fig:minotnbhd}
\end{center}
\end{figure}

Although this ball is overtwisted, the Morse foliated open book is non-left-veering.  To see this, observe in Figure~\ref{fig:minotnbhd} that $P$ consists of eight discs, each of which has a single component of $B$ on the boundary.  Thus there are no non-boundary-parallel arcs in $P$ and the monodromy is trivial.
\end{example}

\

%%%%%%%%%%%%%%%%%%%%%%%%%%%%%%%%%%%%%%%%%%%%%%%%%%%%%%%

\bibliographystyle{alpha}

\bibliography{master}

\newcommand{\etalchar}[1]{$^{#1}$}
\begin{thebibliography}{HKM09b}

\bibitem[AFH{\etalchar{+}}21]{AFHLPV}
Akram Alishahi, Vikt\'{o}ria F\"{o}ldv\'{a}ri, Kristen Hendricks, Joan~E.
  Licata, Ina Petkova, and Vera V\'{e}rtesi.
\newblock Bordered {F}loer homology and contact structures, 2021.

\bibitem[Ben83]{Ben}
Daniel Bennequin.
\newblock Entrelacements et \'{e}quations de {P}faff.
\newblock In {\em Third {S}chnepfenried geometry conference, {V}ol. 1
  ({S}chnepfenried, 1982)}, volume 107 of {\em Ast\'{e}risque}, pages 87--161.
  Soc. Math. France, Paris, 1983.

\bibitem[BM92]{BM}
Joan~S. Birman and William~W. Menasco.
\newblock Studying links via closed braids. {I}. {A} finiteness theorem.
\newblock {\em Pacific J. Math.}, 154(1):17--36, 1992.

\bibitem[Goo05]{Goodman}
Noah Goodman.
\newblock Overtwisted open books from sobering arcs.
\newblock {\em Algebr. Geom. Topol.}, 5:1173--1195, 2005.

\bibitem[HKM07]{HKMrv1}
Ko~Honda, William~H. Kazez, and Gordana Mati\'{c}.
\newblock Right-veering diffeomorphisms of compact surfaces with boundary.
\newblock {\em Invent. Math.}, 169(2):427--449, 2007.

\bibitem[HKM09a]{hkm09}
Ko~Honda, William~H. Kazez, and Gordana Mati\'{c}.
\newblock The contact invariant in sutured {F}loer homology.
\newblock {\em Invent. Math.}, 176(3):637--676, 2009.

\bibitem[HKM09b]{HKM09_HF}
Ko~Honda, William~H. Kazez, and Gordana Mati{\'{c}}.
\newblock On the contact class in {H}eegaard {F}loer homology.
\newblock {\em Journal of Differential Geometry}, 83(2):289--311, October 2009.

\bibitem[IK14a]{IK1}
Tetsuya Ito and Keiko Kawamuro.
\newblock Open book foliation.
\newblock {\em Geom. Topol.}, 18(3):1581--1634, 2014.

\bibitem[IK14b]{IKtot}
Tetsuya Ito and Keiko Kawamuro.
\newblock Visualizing overtwisted discs in open books.
\newblock {\em Publ. Res. Inst. Math. Sci.}, 50(1):169--180, 2014.

\bibitem[LV20]{LV}
Joan~E. Licata and Vera V\'{e}rtesi.
\newblock Foliated open books.
\newblock \url{arXiv:2002.01752}, 2020.

\bibitem[Pav08]{Pav}
Elena Pavelescu.
\newblock {\em Braids and open book decompositions}.
\newblock ProQuest LLC, Ann Arbor, MI, 2008.
\newblock Thesis (Ph.D.)--University of Pennsylvania.

\end{thebibliography}

\end{document}